\documentclass[11pt]{amsart}
\usepackage{amscd}
\usepackage{amsmath}
\usepackage{amsxtra}
\usepackage{amsfonts}
\usepackage{amssymb}

\newtheorem{theorem}{Theorem}[section]

\newtheorem{lemma}[theorem]{Lemma}
\newtheorem{proposition}[theorem]{Proposition}

\theoremstyle{definition}
\newtheorem{definition}[theorem]{Definition}
\newtheorem{remark}[theorem]{Remark}

\newtheorem{example}[theorem]{Example}
\theoremstyle{remark}

\renewcommand{\theclaim}{\textup{\theclaim}}

\numberwithin{equation}{section}

\def\openone

{\mathchoice

{\hbox{\upshape \small1\kern-3.3pt\normalsize1}}

{\hbox{\upshape \small1\kern-3.3pt\normalsize1}}

{\hbox{\upshape \tiny1\kern-2.3pt\SMALL1}}

{\hbox{\upshape \Tiny1\kern-2pt\tiny1}}}

\makeatletter

\newbox\ipbox

\newcommand{\ip}[2]{\left\langle #1\, , \,#2\right\rangle}
\newcommand{\diracb}[1]{\left\langle #1\mathrel{\mathchoice

{\setbox\ipbox=\hbox{$\displaystyle \left\langle\mathstrut #1\right.$}

\vrule height\ht\ipbox width0.25pt depth\dp\ipbox}

{\setbox\ipbox=\hbox{$\textstyle \left\langle\mathstrut #1\right.$}

\vrule height\ht\ipbox width0.25pt depth\dp\ipbox}

{\setbox\ipbox=\hbox{$\scriptstyle \left\langle\mathstrut #1\right.$}

\vrule height\ht\ipbox width0.25pt depth\dp\ipbox}

{\setbox\ipbox=\hbox{$\scriptscriptstyle \left\langle\mathstrut #1\right.$}

\vrule height\ht\ipbox width0.25pt depth\dp\ipbox}

}\right. }

\newcommand{\dirack}[1]{\left. \mathrel{\mathchoice

{\setbox\ipbox=\hbox{$\displaystyle \left.\mathstrut #1\right\rangle$}

\vrule height\ht\ipbox width0.25pt depth\dp\ipbox}

{\setbox\ipbox=\hbox{$\textstyle \left.\mathstrut #1\right\rangle$}

\vrule height\ht\ipbox width0.25pt depth\dp\ipbox}

{\setbox\ipbox=\hbox{$\scriptstyle \left.\mathstrut #1\right\rangle$}

\vrule height\ht\ipbox width0.25pt depth\dp\ipbox}

{\setbox\ipbox=\hbox{$\scriptscriptstyle \left.\mathstrut #1\right\rangle$}

\vrule height\ht\ipbox width0.25pt depth\dp\ipbox}

} #1\right\rangle}

\newcommand{\cj}[1]{\overline{#1}}

\newcommand{\bz}{\mathbb{Z}}

\newcommand{\br}{\mathbb{R}}

\newcommand{\bn}{\mathbb{N}}

\def\blfootnote{\xdef\@thefnmark{}\@footnotetext}

\input xy
\xyoption{all}
\usepackage{amssymb}

\newcommand{\cD}{\mathcal{D}}
\newcommand{\Rs}{{R^*}}

\def\H{\mathcal{H}}

\def\-{^{-1}}

\def\D{\mathcal{D}}

\begin{document}

\title[On the Beurling dimension of exponential frames]{On
the Beurling dimension of exponential frames}
\author{Dorin Ervin Dutkay}
\blfootnote{}
\address{[Dorin Ervin Dutkay] University of Central Florida\\
    Department of Mathematics\\
    4000 Central Florida Blvd.\\
    P.O. Box 161364\\
    Orlando, FL 32816-1364\\
U.S.A.\\} \email{ddutkay@mail.ucf.edu}
\author{Deguang Han}
\address{[Deguang Han] University of Central Florida\\
    Department of Mathematics\\
    4000 Central Florida Blvd.\\
    P.O. Box 161364\\
    Orlando, FL 32816-1364\\
U.S.A.\\} \email{dhan@pegasus.cc.ucf.edu}

\author{Qiyu Sun}
\address{[Qiyu Sun]University of Central Florida\\
    Department of Mathematics\\
    4000 Central Florida Blvd.\\
    P.O. Box 161364\\
    Orlando, FL 32816-1364\\
U.S.A.\\} \email{qsun@mail.ucf.edu}

\author{Eric Weber}
\address{[Eric Weber]Department of Mathematics\\
396 Carver Hall\\
Iowa State University\\
Ames, IA 50011\\
U.S.A.\\} \email{esweber@iastate.edu}

\thanks{}

\begin{abstract}

We study Fourier
frames of exponentials on fractal measures associated with a class of affine
iterated function systems. We prove that, under a mild technical condition,
the Beurling dimension of a Fourier frame coincides with the Hausdorff
dimension of the fractal.
\end{abstract}

\maketitle

%

\tableofcontents

\section{Introduction}
%
%

A family of vectors $(e_i)_{i\in I}$ in a Hilbert space $\H$ is called a {\it frame} if there exist $m,M>0$ such that
$$m\|f\|^2\leq \sum_{i\in I}|\ip{f}{e_i}|^2\leq M\|f\|^2.$$
The constants $m$ and $M$ are called {\it lower and upper bounds} of the frame. If only the upper bound holds, then $(e_i)_{i\in I}$ is called a
{\it Bessel sequence}, and the upper bound is called the {\it Bessel bound}.

Frames provide robust, basis-like (but generally nonunique) representations of vectors in a Hilbert space. The potential redundancy of frames
often allows them to be more easily constructible than bases, and to possess better properties than those that are achievable using bases. For example,
redundant frames offer more resilience to the effects of noise or to erasures of frame elements than bases. Frames were introduced by Duffin
and Schaeffer \cite{DuSc52} in the context of nonharmonic Fourier series, and today they have applications in a wide range of areas. Following
Duffin and Schaeffer a {\it Fourier frame} or {\it frame of exponentials} is a frame of the form $\{e^{2\pi i\lambda\cdot x}\}_{\lambda\in \Lambda}$ for the
Hilbert space $L^{2}[0, 1]$. Fourier frames are also closely connected with sampling sequences or complete interpolating sequences \cite{OSANN}.

The main result of Duffin and Schaeffer is a sufficient density condition for $\{e^{2\pi i\lambda\cdot x}\}_{\lambda\in \Lambda}$ to be a frame. Landau \cite{MR0222554}, Jaffard
\cite{Jaffard} and Seip \cite{Seip2} ``almost"  characterize the frame properties of $\{e^{2\pi i\lambda\cdot x}\}_{\Lambda\in \Lambda}$ in
terms of lower Beurling density:
$$
\mathcal D^-(\Lambda):= \liminf_{h\rightarrow\infty}\inf_{x\in \br}\frac{\#(\Lambda\cap[x-h,x+h])}{2h}.
$$

\begin{theorem}\label{th1.1} For $\{e^{2\pi i\lambda\cdot x}\}_{\Lambda\in \Lambda}$ to be a frame for $L^2[0,1]$, it is necessary that $\Lambda$ is relatively separated
and $\mathcal D^-(\Lambda)\geq 1$,  and it is sufficient that $\Lambda$ is relatively separated and $\mathcal D^-(\Lambda)> 1$.
\end{theorem}

The property of relative separation is equivalent to the condition that the upper density $$\mathcal D^{+}(\Lambda):=
\limsup_{h\rightarrow\infty}\sup_{x\in R}\frac{\#(\Lambda\cap[x-h,x+h])}{2h}$$ is finite.

For the critical case when $\mathcal D^-(\Lambda)= 1$, the complete characterization  was beautifully formulated by Joaquim Ortega-Cerd\`{a} and Kristian
Seip in \cite{OSANN} where the key step is to connect the problem with  de Branges' theory of Hilbert spaces of entire functions, and this
new characterization lead to applications in a classical inequality of H. Landau and an approximation problem for subharmonic functions.

In recent years there has been a wide range of interests in expanding the classical Fourier analysis to fractal or more general probability measures \cite{MR2509326,MR2435649,MR1744572,MR1655831,MR2338387,MR2200934,MR2297038,MR1785282,MR2279556,MR2443273}. One of the central themes of this area of research involves constructive and computational bases in
$L^2(\mu)$-Hilbert spaces, where $\mu$ is a measure which is determined by some self-similarity property. Hence these include classical Fourier
bases, as well as wavelet and frame constructions.

\begin{definition}\label{deff1}
Let $\mu$ be a finite Borel measure on $\br^d$. We say that a set $\Lambda$ in $\br^d$ is a {\it frame spectrum} for $\mu$, with {\it
frame bounds} $m,M>0$, if the set $E(\Lambda):=\{e_\lambda : \lambda\in\Lambda\}$ is a frame for $L^2(\mu)$ with frame
bounds $m$ and $M$. We say that a set $\Lambda$ in $\br^d$ is a {\it Bessel spectrum} for $\mu$ with {\it Bessel bound} $M>0$, if the set $E(\Lambda)$ is a Bessel sequence for $L^2(\mu)$ with Bessel bound $M$. We call $\Lambda$ simply a {\it spectrum} for $\mu$ if $E(\Lambda)$ is an orthonormal basis. If $\mu$ has a spectrum then we say that $\mu$ is a {\it spectral measure}. We will also call $E(\Lambda)$ a {\it Fourier frame/Bessel sequence/orthonormal basis}.
\end{definition}

In \cite{MR1655831} Jorgensen and Pedersen proved the surprising result that for certain Cantor measures it is possible to construct orthonormal Fourier bases. Their motivating example consisted of the Cantor set defined by dividing the unit interval into four equal intervals, keeping only the first and the third, and repeating the process {\it ad inf}. On the resulting Cantor set one considers the appropriate Hausdorff measure of dimension $\ln2/\ln4$, or equivalently the Hutchinson measure as in Definition \ref{deff2} below. They proved that the spectrum of this measure is
$$\Lambda:=\left\{\sum_{k=0}^n 4^k l_k : l_k\in\{0,1\}\right\}.$$

In the same paper, they proved that for the middle-third Cantor set it is impossible to construct more than two mutually orthogonal exponentials, thus no hope for a spectrum. However, it is still unknown whether a frame spectrum exists for the middle-third Cantor measure.

The Jorgensen-Pedersen example naturally generated  questions about the existence (or characterization) of frames of exponentials for Borel probability measures that are not spectral measures. The main focus of this paper are the fractal measures induced by affine iterated function systems. We will be interested in the corresponces between the geometry of the fractal measure and the properties of its Bessel/frame spectra.

\begin{definition}\label{deff1.5}
Let $R$ be a $d\times d$ expanding real matrix, i.e., all eigenvalues $\lambda$ satisfy $|\lambda|>1$. Let $B$ be a finite subset of $\br^d$ of
cardinality $\#B=:N$. For convenience, we can assume $0\in B$. We consider the following {\it affine iterated function system}:
\begin{equation}
\tau_b(x)=R^{-1}(x+b),\quad(x\in\br^d,b\in B).
    \label{eq1_1}
\end{equation}
We denote by $R^*$ the transpose of $R$.
\end{definition}

Since $R$ is expanding, the maps $\tau_b$ are contractions (in an appropriate metric equivalent to the Euclidean one), and therefore Hutchinson's theorem can be applied:
\begin{theorem}\label{th1_1}\cite{Hut81} There exists a unique compact set $X=X_B\subset\br^d$ such that
\begin{equation}
    X=\bigcup_{b\in B}\tau_b(X).
    \label{eq1_2}
\end{equation}
    Moreover
\begin{equation}
    X_B=\left\{\sum_{k=1}^\infty R^{-k}b_k :  b_k\in B\mbox{ for all }k\in\bn\right\}.
    \label{eq1_3}
\end{equation}
There exists a unique Borel probability measure $\mu=\mu_B$ on $\br^d$ such that
\begin{equation}
    \int f\,d\mu=\frac1N\sum_{b\in B}\int f\circ\tau_b\,d\mu,
    \label{eq1_4}
\end{equation}
for all compactly supported continuous functions $f$ on $\br^d$.

Moreover, the measure $\mu_B$ is supported on the set $X_B$.
\end{theorem}

It is the measure $\mu_B$ that we will focus on here, and we will give necessary and/or sufficient conditions for a set $\Lambda$ to be a Bessel/frame spectrum for $L^2(\mu_B)$.

\begin{definition}\label{deff2}
The set $X_B$ in \eqref{eq1_2},\eqref{eq1_3} is called {\it the attractor} of the iterated function system (IFS) $(\tau_b)_{b\in B}$. The
measure $\mu_B$ in \eqref{eq1_4} is called {\it the invariant measure} of the IFS $(\tau_b)_{b\in B}$. We will also say that $\mu_B$ is an {\it affine IFS measure}.

If $\mu_B(\tau_b(X_B)\cap \tau_{b'}(X_B))=0$ for all $b\neq b'$, then we say that the IFS $(\tau_b)_{b\in B}$ (or the measure $\mu_B$) {\it has
no overlap}.
\end{definition}

To establish a connection between the Hausdorff dimension of the fractal and the properties of its frame or Bessel spectra, we will use the notion of Beurling dimension, or more precisely the upper Beurling dimension (Definition \ref{deff3}). This notion was introduced in \cite{CKS08} for the study of irregular Gabor frames.

The paper is organized as follows: the main result of the paper is Theorem \ref{thf4}. It shows that for affine IFS measures with no overlap, any frame spectrum satisfying some mild assumptions will have the Beurling dimension equal to the Hausdorff dimension of the fractal. The formula for the Hausdorff dimension of the fractal (see \cite{Fal}) is $\log_\rho N$ where $N$ is the number of contractions in the affine iterated function systems and $\rho$ is the contraction constant (see Remark \ref{rem3.4} for more details).

 In section \ref{gen} we present some preliminaries. Proposition \ref{pri6} shows that oversampling by a factor of $R$ preserves the frame property and increases the frame bounds by a factor of $N$. Proposition \ref{prst} shows that the Bessel spectra are stable under uniformly bounded perturbation, and that frame spectra are stable under small uniformly bounded deformations.

\section{Preliminaries}\label{gen}
We begin with an oversampling result. It can be formulated briefly as: oversampling by $R$ implies multiplying the bounds by $N$.

\begin{proposition}\label{pri6}
Suppose the IFS $(\tau_b)_{b\in B}$ has no overlap and let $\mu=\mu_B$ be its invariant measure. Suppose $\Lambda$ is a frame (Bessel) spectrum for
$\mu$. Then
\begin{enumerate}
\item
$\Rs^{-1}\Lambda$ is a frame (Bessel) spectrum for $\mu$, and the frame bounds are multiplied by $N$.
\item
For every $n\geq0$ the following inequality holds
\begin{equation}\label{eqf1.1}
m N^n\leq \sum_{\lambda\in\Lambda}\left|\widehat\mu_B\left(x-{R^*}^{-n}\lambda\right)\right|^2\leq MN^n,\quad(x\in\br^d).
\end{equation}

The hat denotes the Fourier transform of the measure
$$\widehat\mu(t)=\int e^{2\pi it\cdot x}\,d\mu(x),\quad(t\in\br^d).$$
\end{enumerate}
\end{proposition}

\begin{proof}
First we need a Lemma.
\begin{lemma}\label{pri5}
If the measure $\mu_B$ has no overlap, then for all $\mu_B$-integrable functions $g$, and any $b_0\in B$:
$$\int g\circ\tau_{b_0}\,d\mu_B=N\int_{\tau_{b_0}(X_B)}g\,d\mu_B.$$
\end{lemma}
\begin{proof}
From the invariance equation
$$N\int\chi_{\tau_{b_0}(X_B)}g\,d\mu_B=\sum_{b\in B}\int\chi_{\tau_{b_0}(X_B)}\circ\tau_b\,g\circ\tau_b\,d\mu_B.$$
Since there is no overlap $\chi_{\tau_{b_0}(X_B)}\circ\tau_{b}$ is 1 if $b=b_0$, and is 0 if $b\neq b_0$, $\mu_B$-almost everywhere (since
$\mu_B$ is supported on $X_B$).

This leads to the conclusion.
\end{proof}

We return to the proof of Proposition \ref{pri6}. (i) We assumed $0\in B$. Take $g$ in $L^2(\mu)$. Then
$$\ip{g}{e_{\Rs^{-1}\lambda}}=\int g(x) \cj{e_{\Rs^{-1}\lambda}(x)}\,d\mu(x)=\int g(x)\cj{e_\lambda(R^{-1}x)}\,d\mu(x)
=\int g\circ\tau_0^{-1}\circ\tau_0\,\cj e_\lambda\circ\tau_0\,d\mu$$ and using Lemma \ref{pri5},
$$=N\int_{\tau_0(X_B)} g\circ \tau_0^{-1}\cj e_\lambda\,d\mu=N\ip{\chi_{\tau_0(X_B)}g\circ\tau_0^{-1}}{e_\lambda}$$
Then
$$A_g:=\sum_{\lambda\in\Lambda}|\ip{g}{e_{\Rs^{-1}\lambda}}|^2=N^2\sum_{\lambda\in\Lambda}|\ip{\chi_{\tau_0(X_B)}g\circ\tau_0^{-1}}{e_\lambda}|^2.$$
Using the frame bounds for $\Lambda$ we obtain that
$$N^2m \|\chi_{\tau_0(X_B)}g\circ\tau_0^{-1}\|^2\leq A_g\leq N^2M\|\chi_{\tau_0(X_B)}g\circ\tau_0^{-1}\|^2.$$
Using Lemma \ref{pri5} again, we have that
$$N\|\chi_{\tau_0(X_B)}g\circ\tau_0^{-1}\|^2=N\int_{\tau_0(X_B)}|g|^2\circ\tau_0^{-1}\,d\mu=\int |g|^2\,d\mu=\|g\|^2.$$
This yields the conclusion.

By (i) it is enough to prove \eqref{eqf1.1} for $n=0$. But this follows by an application of the frame inequalities to the
function $e_x$, and using the fact that $\ip{e_x}{e_\lambda}=\widehat\mu_B(x-\lambda)$.
\end{proof}

The next proposition shows that, similar to the celebrated Paley-Wiener theorem \cite[Theorem XXXVII]{MR1451142} the Bessel property is invariant under uniformly bounded perturbations of the Bessel spectrum and the frame
property is stable under small perturbation of the frame spectrum. The proof is a small modification of the proofs of Lemma II and III in \cite{DuSc52}.
\begin{proposition}\label{prst}
Let $\mu$ be a compactly supported Borel probability measure on $\br^d$. Let $\Lambda=\{\lambda_n : n\in\bn\}$ and
$\Gamma=\{\mu_n : n\in\bn\}$ be two subsets of $\br^d$ with the property that there exists an $L>0$ such that
$$|\lambda_n-\mu_n|\leq L,\quad(n\in\bn)$$
\begin{enumerate}
\item If $\Lambda$ is a Bessel spectrum for $\mu$ then $\Gamma$ is one too. Moreover, the Bessel bound for $\Gamma$ depends only on the Bessel bound for $\Lambda$, on $L$ and on the size of the support of $\mu$.

 \item If $\Lambda$ is a frame spectrum for $\mu$, then there exists
a $\delta>0$ such that if $L\leq \delta$ then $\Gamma$ is a frame spectrum too. Moreover, $\delta$ depends only on the frame bounds of $\Lambda$
and on the size of the support of $\mu$.
\end{enumerate}
\end{proposition}

\begin{proof} (i) It is enough to prove the assertion for the case when for all $n\in\bn$, $\mu_n$ differs from $\lambda_n$ only on the first component, because then the statement follow by induction on the number of components.

For $x\in\br^d$ we denote by $(x_1,\dots,x_d)$ its components.

Let $f\in L^\infty(\mu)$ and define by
$$\tilde f(t):=\ip{f}{e_t}=\int f(x)e^{-2\pi it\cdot x}\,d\mu(x),\quad(t\in\br^d).$$
Clearly $\tilde f$ is analytic in each variable $t_1,\dots,t_d$. Also
\begin{equation}
\frac{\partial^k\tilde f}{\partial t_1^k}(t)=\int f(x)(-2\pi ix_1)^ke^{-2\pi it\cdot x}\,d\mu(x)=\ip{(-2\pi ix_1)^kf}{e_t},\quad(t\in\br^d).
\label{eqstde1}
\end{equation}
We have for all $n\in\bn$, using the Taylor expansion at $(\lambda_n)_1$ in the first variable:
$$|\tilde f(\mu_n)-\tilde f(\lambda_n)|^2=\left|\sum_{k=1}^\infty\frac{\frac{\partial^k\tilde f}{\partial t_1^k}(\lambda_n)}{k!}((\mu_n)_1-(\lambda_n)_1)^k\right|^2$$
and using the Cauchy-Schwarz inequality
$$\leq \sum_{k=1}^\infty\frac{\left|\frac{\partial^k\tilde f}{\partial t_1^k}(\lambda_n)\right|^2}{k!}\cdot\sum_{k=1}^\infty\frac{|(\mu_n)_1-(\lambda_n)_1|^{2k}}{k!}
\leq \sum_{k=1}^\infty\frac{\left|\frac{\partial^k\tilde f}{\partial t_1^k}(\lambda_n)\right|^2}{k!}\cdot\sum_{k=1}^\infty\frac{L^{2k}}{k!}$$
$$=\sum_{k=1}^\infty\frac{\left|\frac{\partial^k\tilde f}{\partial t_1^k}(\lambda_n)\right|^2}{k!}\cdot(e^{L^2}-1).$$

Also, since $\Lambda$ is a Bessel spectrum with Bessel bound $B$,
$$\sum_{n\in\bn}\left|\frac{\partial^k\tilde f}{\partial t_1^k}(\lambda_n)\right|^2=\sum_{n\in\bn}\left|\ip{(-2\pi ix_1)^kf}{e_{\lambda_n}}\right|^2
\leq B\|(-2\pi ix_1)^kf\|^2\leq B(2\pi M)^{2k}\|f\|^2$$ where $M$ is picked in such a way that the support of $\mu$ is contained in the ball
$B(0,M)$.

Using these inequalities and interchanging the order of summation, we obtain
\begin{equation}
\sum_{n\in\bn}|\tilde f(\mu_n)-\tilde f(\lambda_n)|^2\leq B(e^{L^2}-1)\|f\|^2\sum_{k=1}^\infty\frac{(2\pi
M)^{2k}}{k!}=B(e^{L^2}-1)(e^{M^2}-1)\|f\|^2. \label{eqst2}
\end{equation}
Using Minkowski's inequality we obtain then
$$\left(\sum_{n}|\tilde f(\mu_n)|^2\right)^{1/2}\leq\left(\sum_{n}|\tilde f(\lambda_n)|^2\right)^{1/2}+\left(\sum_n|\tilde f(\mu_n)-\tilde f(\lambda_n)|^2\right)^{1/2}$$$$
\leq \left(B^{1/2}+(B(e^{L^2}-1)(e^{M^2}-1))^{1/2}\right)\|f\|$$ and this implies that $\Gamma$ is a Bessel spectrum.

(ii) We saw in (i) that the Bessel property is preserved. Let $A,B$ be the lower and upper frame bounds for $\Lambda$. Pick $\delta>0$ small
enough so that
$$A^{1/2}-(B(e^{L^2}-1)(e^{M^2}-1))^{1/2}>0$$
for $0<L\leq\delta$. Then, using \eqref{eqst2} and Minkowski's inequality we have
$$\left(\sum_{n}|\tilde f(\mu_n)|^2\right)^{1/2}\geq\left(\sum_{n}|\tilde f(\lambda_n)|^2\right)^{1/2}-\left(\sum_n|\tilde f(\mu_n)-\tilde f(\lambda_n)|^2\right)^{1/2}$$
$$\geq\left( A^{1/2}-(B(e^{L^2}-1)(e^{M^2}-1))^{1/2}\right)\|f\|.$$
So $\Gamma$ is a frame. Note that $\delta$ depends only on $A,B$ and $M$, not on $\mu$ or $\Lambda$.
\end{proof}

\section{Beurling dimension}\label{beu}
\begin{definition}\label{deff3}
\cite{CKS08} Let $\Lambda$ be a discrete subset of $\br^d$. For $r>0$, the {\it upper Beurling density corresponding to $r$} (or {\it $r$-Beurling density}) is defined by
$$\mathcal D_r^+(\Lambda):=\limsup_{h\rightarrow\infty}\sup_{x\in\br^d}\frac{\#(\Lambda\cap(x+h[-1,1]^d))}{h^r}.$$
The {\it upper Beurling dimension} (or simply the {\it Beurling dimension}) is defined by
$$\dim^+(\Lambda):=\sup\{r>0 : \D^+_r(\Lambda)>0\}.$$
Alternatively,
$$\dim^+(\Lambda)=\inf\{r>0 : \D_r^+(\Lambda)<\infty\}.$$

Given a set of exponential functions $E(\Lambda):=\{e_\lambda : \lambda\in\Lambda\}$ we also say that $\mathcal D_r^+(\Lambda)$ is the $r$-Beurling density of $E(\Lambda)$.
\end{definition}

\begin{definition}\label{deff4}
We say that a $d\times d$ real matrix $R$ is a {\it similarity with scaling factor} $\rho>0$ if there exists an orthogonal matrix $O$ such that
$R=\rho O$.

\end{definition}

\begin{remark}
Note that if $R$ is a similarity with scaling factor $\rho$, then so is $R^*$, and $\|R^*x\|=\rho\|x\|$ for all $x\in\br^d$, and $R^* \cj
B(0,r)=\rho \cj B(0,r)$ for all $r>0$, where $\cj B(0,r)$ is the closed ball of center $0$ and radius $r$.

\end{remark}

\begin{remark}\label{rem3.4}
Suppose $R$ is a similarity with scaling costant $\rho$, and let $X_B$ be the attractor of the affine iterated function system $(\tau_b)_{b\in B}$. Assume in addition that the {\it open set condition} is satisfied, i.e., there exists a non-empty open set $V$ such that
$$V\supset\bigcup_{b\in B}\tau_b(V).$$
Then Theorem 9.3 in \cite{Fal} shows that the Hausdorff dimension of $X_B$ is $s:=\log_\rho N$. Moreover, for this value, the Hausdorff measure of $X_B$ satisfies $0<\H^s(X_B)<\infty$.

Thus the next theorem will prove that, under these conditions, the Beurling density of a frame spectrum for the invariant measure $\mu_B$ will be equal to the Hausdorff measure of the attractor $X_B$.

The Hausdorff measure has the property that $\H^s(R\cdot E)=N\H^s(E)$ for all Borel subsets $E$, since $R$ is a similarity. Also $\H^s(E+t)=\H^s(E)$ for all $t\in\br^d$. Therefore, it is easy to check that the restriction of the Hausdorff measure $\H^s$ to $X_B$ satisfies the invariance equation \eqref{eq1_4}, under these conditions (which guarantee also that there is no overlap). Therefore, since the invariant measure is unique, it follows that $\mu_B$ is the restriction of the Hausdorff measure $\H^s$ to $X_B$, renormalized so that it is probability measure.

\end{remark}

\begin{theorem}\label{thf4}
Let $\mu=\mu_B$ be the invariant measure of the affine IFS $(\tau_b)_{b\in B}$, and assume $R$ is a similarity with scaling constant $\rho>1$.

(a) If $\Lambda$ is a Bessel spectrum for $\mu$ then its upper Beurling dimension satisfies

\begin{equation}
    \dim^+(\Lambda)\leq\log_\rho N.
    \label{eqf6}
\end{equation}
Moreover
\begin{equation}
    \cD_{\log_\rho N}^+(\Lambda)<\infty.
    \label{eqf7}
\end{equation}

(b) Assume in addition that the following condition is satisfied: there exists a natural number $p\geq 1$, such that
\begin{equation}
    \sup_{\lambda\in\Lambda}\operatorname*{dist}\left(\Rs^{-p}\lambda,\Lambda\right)<\infty.
    \label{eqfb1}
\end{equation}
Then, if $\Lambda$ is a frame spectrum for $\mu$, the upper Beurling dimension is
\begin{equation}
    \dim^+(\Lambda)=\log_\rho N=\mbox{\textup{Hausdorff dimension}}(X_B).
    \label{eqfb2}
\end{equation}
\end{theorem}

\begin{proof} We first prove  \eqref{eqf7}.

Let $U:=\cj B(0,r_0)$, the closed ball centered at $0$ of radius $r_0$, where $r_0$ is picked in such a way that the Lebesgue measure of $U$ is $2^d$. By
\cite[Proposition 2.2]{CKS08}, one can compute the Beurling densities using the set $U$: for $r>0$,

$$\cD_r^+(\Lambda)=\limsup_{h\rightarrow\infty}\sup_{x\in\br^d}\frac{\#(\Lambda\cap(x+hU))}{h^r}.$$

Since $\widehat\mu_B$ is continuous and $\widehat\mu_B(0)=1$, we can pick $0<\epsilon<1$ and $\delta>0$ such that
$$|\widehat\mu_B(x)|^2\geq\epsilon\mbox{ for all }x\in \delta U.$$
Take $h>\delta$ arbitrary. Take the first $n\in\bn$ such that $\rho^{-n}h\leq\delta$. Then we also have $\rho^{-n}h\geq \delta/\rho$, and
$\rho^{-n}hU\subset \delta U$.

Now take an arbitrary $x\in\br^d$. We have, using the fact that $R^*$ is a similarity so it scales distances:
$$\epsilon\cdot\#(\Lambda\cap(x+hU))=\epsilon\cdot\#(({R^*}^{-n}(\Lambda-x))\cap(\Rs^{-n}(hU)))=\epsilon\cdot\#(({R^*}^{-n}(\Lambda-x))\cap(\rho^{-n}hU))$$
$$\leq \sum_{\lambda\in\Lambda, \Rs^{-n}(\lambda-x)\in \rho^{-n}hU}\left|\widehat\mu_B\left(\Rs^{-n}(\lambda-x)\right)\right|^2\leq\sum_{\lambda\in\Lambda}\left|\widehat\mu_B\left(\Rs^{-n}(\lambda-x)\right)\right|^2\leq BN^n,$$
where we have used Proposition \ref{pri6}(ii) in the last inequality.

On the other hand, since $\rho^{-n}h\geq\delta/\rho$, we have that $h\geq \rho^n\delta/\rho$ so
$$\frac{\epsilon\cdot\#(\Lambda\cap(x+hU))}{h^{\log_\rho N}}\leq \frac{BN^n}{\left(\frac{\delta}{\rho}\right)^{\log_\rho N}N^n}=:C.$$

But this shows that $\cD_{\log_\rho N}^+(\Lambda)\leq C/\epsilon$ so \eqref{eqf7} is proved, and $\dim^+(\Lambda)\leq \log_\rho N$.

We prove (b)  by contradiction. Assume $\dim^+(\Lambda)<\log_\rho N$. Pick $\beta<N$ such that $\dim^+(\Lambda)<\log_\rho\beta$. Then by the definition of the Beurling dimension
$\cD_{\log_\rho\beta}^+(\Lambda)=0$. Therefore there exists $h_0>0$ such that
\begin{equation}
    \#(\Lambda\cap (x+h U))\leq h^{\log_\rho\beta},\mbox{ for }h\geq h_0, x\in\br^d.
    \label{eqfb8}
\end{equation}

The assumption \eqref{eqfb1} implies that there exists $C>0$ and two functions $x:\Lambda\rightarrow \cj B(0,C)$,
$\gamma:\Lambda\rightarrow\Lambda$ such that
\begin{equation}
    \Rs^{-p}\lambda=x(\lambda)+\gamma(\lambda)\mbox{ for all }\lambda\in\Lambda.
    \label{eqfb9}
\end{equation}

Iterating \eqref{eqfb9} we get for all $n$:
$$\Rs^{-np}\lambda=\Rs^{-(n-1)p} x(\lambda)+\dots+ x(\gamma^{n-1}(\lambda))+\gamma^n(\lambda)=:x_n(\lambda)+\gamma^n(\lambda),$$
where by $\gamma^n$ we mean $\gamma$ composed with itself $n$ times.

Since $\Rs$ is a similarity, we have
$$\|x_n(\lambda)\|\leq C(\rho^{-(n-1)p}+\dots+1)\leq D$$
for some constant $D$.

Thus
\begin{equation}
    \Rs^{-np}\lambda=x_n(\lambda)+\gamma^n(\lambda),\mbox{ with }\|x_n(\lambda)\|\leq D\mbox{ and }\gamma^n(\lambda)\in\Lambda\mbox{ for all }\lambda\in\Lambda.
    \label{eqfb10}
\end{equation}

We claim that there is a constant $E$ such that for $n$ big enough
\begin{equation}
    \#\left\{\lambda\in\Lambda : \gamma^n(\lambda)=\lambda'\right\}\leq E\beta^{np},\mbox{ for all }\lambda'\in\Lambda.
    \label{eqfb11}
\end{equation}

Fix $\lambda'\in\Lambda$. For each $n$, let $s_n$ be the number of $\lambda$ such that $\gamma^n(\lambda)=\lambda'$. Let
$\lambda_1,\dots,\lambda_{s_n}$ be the list of such $\lambda$. Then, by \eqref{eqfb10},  $\Rs^{-np}\lambda=x_n(\lambda_i)+\lambda'$ for all $i$.
But then
$$\|\lambda_i-\lambda_1\|\leq \rho^{np}\|x_n(\lambda_i)-x_n(\lambda_1)\|\leq 2 D\rho^{np}.$$
Therefore $\lambda_i\in\lambda_1+(2D/r_0)\rho^{np} U$. If we take $n$ large enough, we can use \eqref{eqfb8} and obtain that $s_n\leq
((2D/r_0)\rho^{np})^{\log_\rho\beta}=:E\beta^{np}$. This implies \eqref{eqfb11}.

Next, we will need the following key lemma:

\begin{lemma}\label{lemi10}
For every $r>0$
\begin{equation}\label{eqi10.1}
\sum_{\lambda\in\Lambda}\sup_{\|x\|\leq r}|\widehat\mu(x+\lambda)|^2<\infty.
\end{equation}

\end{lemma}

\begin{proof}
For each $\lambda\in\Lambda$ pick an $x_\lambda\in \br^d$, with $\|x_\lambda\|\leq r$. Then, by Proposition \ref{prst}(i) we have that $\{\lambda+ x_\lambda : \lambda\in\Lambda\}$ is a Bessel frame with Bessel bound $M$ which does not depend on $\Lambda$ or $\{x_\lambda\}$ but only on the Bessel bound of $\Lambda$, on $r$ and on the size of the support of the measure, i.e., $X_B$. Writing the Bessel inequality for the constant function $1$, we get
$$\sum_{\lambda}|\widehat\mu(\lambda+x_\lambda)|^2\leq M.$$
Since $M$ does not depend on $\{x_\lambda\}$, we can let $x_\lambda$ vary, one by one, and replace
$|\widehat\mu(\lambda+x_\lambda)|^2$ by the corresponding supremum, thus obtaining Lemma \ref{lemi10}.
\end{proof}

Returning to the proof of the theorem, we apply Proposition \ref{pri6}(ii) to $x=0$ and we have, for $n$ big enough
$$mN^{np}\leq \sum_{\lambda\in\Lambda}|\widehat\mu(-\Rs^{-np}\lambda)|^2=\sum_{\lambda\in\Lambda}|\widehat\mu(\Rs^{-np}\lambda)|^2=\sum_{\lambda\in\Lambda}|\widehat\mu(x_n(\lambda)+\gamma^n(\lambda))|^2$$
and using \eqref{eqfb10}
$$=\sum_{\lambda'\in\Lambda}\sum_{\lambda\in\Lambda\, :\, \gamma^n(\lambda)=\lambda'}|\widehat\mu(x_n(\lambda)+\lambda')|^2\leq
\sum_{\lambda'}\sum_{\gamma^n(\lambda)=\lambda'}\sup_{\|x\|\leq D}|\widehat\mu(x+\lambda')|^2$$ and with \eqref{eqfb11}
$$\leq \sum_{\lambda'} E\beta^{np}\sup_{\|x\|\leq D}|\widehat\mu(x+\lambda')|^2.$$
But Lemma \ref{lemi10} shows that
$$\sum_{\lambda'}\sup_{\|x\|\leq D}|\widehat\mu(x+\lambda')|^2\leq F$$
for some finite constant $F$. And this implies $m N^{np}\leq EF\beta^{np}$ which contradicts the fact that $\beta<N$.
\end{proof}

\begin{remark}
We believe that condition \eqref{eqfb1} can be removed from the hypothesis of Theorem \ref{thf4}(b), but we were not able to do it. However, all examples of frame spectra that we know, do satsify this condition (see e.g., \cite{MR2509326,DJ06,DJ07d}). Also, for the classical case of the unit interval, if the lower Beurling density of $\Lambda$ is positive (as in Theorem \ref{th1.1}) then every interval of big enough length will contain some element of $\Lambda$, and therefore condition \eqref{eqfb1} will be clearly satsified.

\end{remark}

\begin{remark}
While the Beurling dimension seems to be a good invariant for frame spectra, the Beurling density can vary a lot. It was proved in \cite{DJ10} that for the Jorgensen-Pedersen Cantor set, where $R=4$, $B=\{0,2\}$, the sets 
$$\Lambda(5^k):=5^k\left\{\sum_{i=0}^n4^il_i : l_i\in\{0,1\}, n\in\bz+\right\},\quad (k\in\bz_+)$$
are all spectra, so the corresponding exponentials form orthonormal bases, not just frames!

The Beurling dimension for all the sets $\Lambda(5^k)$ in this example is $\ln 2/\ln 4=1/2$, and the upper $1/2$-Beurling desity for $\Lambda(5^0)$ can be checked to be positive. It is also finite, by Theorem \ref{thf4}(a). But since $\Lambda(5^k)=5^k\Lambda(5^0)$ it follows that the upper $1/2$-Beurling densities of the sets $\Lambda(5^k)$ decrease to 0. Thus one can have spectra of arbitarily small Beurling density.
\end{remark}

\begin{example}\label{ex2}
As we have explained in the introduction, for the classical case of the unit interval, the condition that the upper Beurling density is finite (for dimension equal to one), is also a sufficient condition for the set to be a Bessel spectrum for the Lebesgue measure on the unit interval. The next example will show that the situation is very different in the case of fractal measures.

Consider the IFS for the middle-third Cantor set, i.e., $R=3$, $B=\{0,2\}$, and let $\mu_3$ be its invariant measure.
\begin{proposition}\label{prex2}
For any integer $a\in\bz\setminus\{0\}$, and any infinite set of non-negative integers $F$, the set $\{3^na : n\in F\}$ cannot be a Bessel
spectrum for the middle-third Cantor measure $\mu_3$.

Moreover, any such set has Beurling dimension $0$.

\end{proposition}

\begin{proof}

First we prove that $\widehat\mu_3(a)\neq 0$. For this, note that the Fourier transform of the invariance equation \eqref{eq1_4} implies
\begin{equation}\label{eqex2_1}
\widehat\mu_3(x)=m_B(\Rs^{-1}x)\widehat\mu_3(\Rs^{-1}x),\quad(x\in\br)
\end{equation}
where
$$m_B(x)=\frac{1}{2}(1+e^{2\pi i2x}).$$
Iterating this relation and taking the limit at infinity we get
$$\widehat\mu_3(x)=\prod_{n=1}^\infty m_B(\Rs^{-n}x),\quad(x\in\br)$$
and the infinite product is uniformly convergent on compact sets. See also \cite{DJ07d,DJ06} for more details.

Then the zeros of $\widehat\mu_3$ are of the form $3^n\frac{(2k+1)\pi}{4}$, $k\in\bz$, $n\geq1$. Thus $\widehat\mu_3(a)\neq0$.

Also, since $m_B(3^ka)=1$ if $a\in\bz$ and $k\geq 0$, the refinement equation \eqref{eqex2_1} implies that $\widehat\mu_3(3^na)=\widehat\mu_3(a)\neq0$ for
all $n\geq 1$.

But then, $\{3^na : n\in F\}$ cannot be a Bessel spectrum if $F$ is infinite, because one has
$$\sum_{n\in F}|\ip{1}{e_{3^na}}|^2=\sum_{n\in F}|\widehat\mu_3(3^na)|^2=\# F\cdot |\widehat\mu_3(a)|^2=\infty.$$

To show that the Beurling dimension is zero, we see that the Beurling dimension of $\{3^na : n\in F\}$ is the same as the one for
$\{3^n : n\in F\}$ and is less than that of $\Lambda:=\{3^n :  n\geq0\}$. Then $\#\Lambda\cap(x+h(-1,1))$ is less than the biggest $p$ such
that $3^n,\dots, 3^{n+p} \in (x-h,x+h)$, for some $n\geq 0$. Then $|3^{n+p}-3^n|\leq 2h$ so $3^p\leq (2h+1)$, which means that $p\leq
\log_3(2h+1)$. But then, for $\alpha>0$ the $\alpha$-Beurling density is less than $\liminf_h  \log_3(2h+1)/h^\alpha=0$. Since $\alpha>0$ was
arbitrary, this shows that the Beurling dimension is zero.
\end{proof}

\end{example}

\end{document}